\documentclass[12pt,twoside,a4paper]{amsart}

\usepackage{mathrsfs} 
\usepackage{amssymb} 
\usepackage{tikz} 
\usetikzlibrary{cd}
\usepackage{xspace} 
\usepackage{enumerate}

\usepackage{microtype}

\newcommand{\bQ}{\mathbb{Q}}
\newcommand{\bR}{\mathbb{R}}

\newcommand{\bA}{\mathbf{A}} 
\newcommand{\bB}{\mathbf{B}}
\newcommand{\bU}{\mathbf{U}}
\newcommand{\bV}{\mathbf{V}}

\newcommand{\bM}{\mathbf{M}}
\newcommand{\bN}{\mathbb{N}}

\newcommand{\bC}{\mathbf{C}}

\newcommand{\cM}{\mathcal{M}}

\newcommand{\restr}{\mathord{\upharpoonright}}

\newcommand{\ta}{{\bar{a}}}
\newcommand{\tb}{{\bar{b}}}

\newcommand{\ty}{\bar{y}}

\newcommand{\Age}{\operatorname{Age}}

\newcommand{\Emb}{\operatorname{Emb}}
\newcommand{\EEmb}{\operatorname{EEmb}}

\newcommand{\Aut}{\operatorname{Aut}}

\newcommand{\Sym}{\operatorname{Sym}}

\newcommand{\fix}{\operatorname{fix}}

\newcommand{\Fraisse}{Fra\"\i{}ss\'e}

\newcommand{\injto}{\hookrightarrow}

\newcommand{\im}{\operatorname{im}}
\newcommand{\Th}{\operatorname{Th}}
\newcommand{\tp}{\operatorname{tp}}

\theoremstyle{plain} \newtheorem{theorem}{Theorem}[section]
\newtheorem{proposition}[theorem]{Proposition}
\newtheorem{lemma}[theorem]{Lemma}
\newtheorem{corollary}[theorem]{Corollary} 
\theoremstyle{definition}
\newtheorem{definition}[theorem]{Definition}
 
\newtheorem{example}[theorem]{Example}

\theoremstyle{remark}

\author[Ch.\,Pech]{Christian Pech}

\address{}
\email{cpech@freenet.de}
\urladdr{}

\author[M.\,Pech]{Maja Pech} \address{Department of
  Mathematics\\University of Novi Sad} 
\email{maja@dmi.uns.ac.rs}
\urladdr{http://people.dmi.uns.ac.rs/~maja/}
\title[Reconstructing the topology\dots]{Reconstructing the topology of the elementary self-embedding monoids of countable saturated structures}
\date{\today}
\begin{document}
	\begin{abstract}
		Every transformation monoid comes equipped with a canonical topology---the topology of pointwise convergence. For some structures, the topology of the endomorphism monoid can be reconstructed from its underlying abstract monoid. This phenomenon is called \emph{automatic homeomorphicity}.
		
		In this paper we show that whenever the automorphism group of a countable saturated structure has automatic homeomorphicity and  a trivial center, then the monoid of elementary self-embeddings has automatic homeomorphicity, too.  
		
		As a second result we strengthen a result by Lascar by showing that whenever $\bA$ is a countable $\aleph_0$-categorical $G$-finite structure whose automorphism group has a trivial center and if $\bB$ is any other countable structure, then every isomorphism between the monoids of elementary self-embeddings is a homeomorphism. 
	\end{abstract}

	\maketitle

\section{Automatic homeomorphicity}
If we  consider a set $A$ as a discrete topological space, then the set of self-mappings $A^A$ of $A$ is naturally equipped with the product topology. A sub-basis of this topology is given by
$\{\Phi_{a,b}\mid a,b\in A\}$, where $\Phi_{a,b}$ is given by
\[ \Phi_{a,b}=\{f\in A^A\mid f(a)=b\}.\]
This topology is also known as the Tychonoff-topology or topology of pointwise convergence. 

The set $A^A$, together with the composition operation is called the full transformation monoid on $A$. We denote it by $T_A$. Every submonoid $M$ of $T_A$ is  naturally equipped with the subspace-topology (in particular, the composition on $M$ is continuous with respect to this topology). 

When using the predicates \emph{open} or \emph{closed} in connection with transformation monoids we make an implicit distinction between transformation monoids in general and the special case of permutation groups. A transformation monoid on $A$ is called \emph{closed} (\emph{open}) whenever it is closed (open) as a subspace of $A^A$. On the other hand, a permutation group is called \emph{closed} (\emph{open}) if it is closed (open) as a subspace of the space $\Sym(A)$ of all permutations on $A$. 

It is well-known that a permutation group is closed if and only if it is the automorphism group of a first order structure. Similarly, a transformation monoid is closed if and only if it is the endomorphism monoid of such a structure.  That is the reason why we are mainly concerned with closed permutation groups and closed transformation monoids. 

As we saw, every transformation monoid determines a topological monoid and this, in turn, gives rise to an abstract monoid. In each step information about the original transformation monoid is lost. The question is, how much information is preserved? The answer depends on the scope of our observation. E.g., it was shown in \cite{BodPin15} that two oligomorphic permutation groups of countable degree are topologically isomorphic if and only if the two permutations groups, when considered in a suitable way as $G$-sets, generate the same pseudo-variety. Thus, from topological knowledge we obtain very concrete information about the action. In other examples, topological information is reconstructed from the completely abstract algebraic level. E.g., every group-isomorphism of the symmetric group on $\bN$ to another closed permutation group of countable degree is a homeomorphism. This is a direct consequence from the observation in \cite{DixNeuTho86}, that $\Sym(\bN)$ has the small-index property (cf. \cite{Ros09b}). This phenomenon is called \emph{automatic homeomorphicity}:
\begin{definition}
	A transformation monoid (a permutation group) is said to have \emph{automatic homeomorphicity with respect to a class $\cM$} of transformation monoids (permutation groups) if every isomorphism to a member of $\cM$ is a homeomorphism.
\end{definition}
If $G$ is a closed permutation group on a countable set $A$, and if $G$ has automatic homeomorphicity with respect to the set of all closed permutation groups on $A$, then we will simply say, that $G$ has \emph{automatic homeomorphicity}. Similarly, if $M$ is a closed transformation monoid on $A$ and if $M$ has automatic homeomorphicity with respect to the set of all closed transformation monoids on $A$, then we say just that $M$ has \emph{automatic homeomorphicity}.

For groups the notion of automatic homeomorphicity (and, more generally, the notion of automatic continuity) has been studied in one or the other guise for a long time. The richest source of permutation groups with automatic homeomorphicity is given by the permutation groups that satisfy the small index property. Another, quite different approach to reconstructing the topology of a permutation group from its abstract group structure are Rubin's (weak) $\forall\exists$-interpretations. A good overview about the existing literature in this area of research is given by the surveys \cite{Ros09b} and \cite{Mac11}. 

Automatic homeomorphicity for transformation monoids was first considered in \cite{BodPinPon17}. Some examples of transformation monoids with automatic homeomorphicity include:
\begin{itemize}
	\item the monoid of injective functions on $\bN$ (\cite{BodPinPon17}),
	\item the full transformation monoid $T_{\bN}$ (\cite{BodPinPon17}),
	\item the self-embedding monoid of the Rado-graph, (\cite{BodPinPon17}),
	\item the endomorphism monoid of the Rado-graph (\cite{BodPinPon17}),
	\item the self-embedding monoid of the countable universal homogeneous digraph (\cite{BodPinPon17}),
	\item the endomorphism monoid of $(\bQ,<)$ (\cite{BehTruVar17}),
	\item the endomorphism monoid of $(\bQ,\le)$ (\cite{BehTruVar17}),
\end{itemize} 
Recently, also the automatic homeomorphicity for the self-embedding monoids and endomorphism monoids of $(\bQ,\operatorname{betw})$, $(\bQ,\operatorname{circ})$, and $(\bQ,\operatorname{sep})$ were established (cf.~\cite{TruVar16}).

Additionally, in \cite{PecPec15b} it was shown that the monoids of  non-expansive selfmaps of the rational Urysohn space and of the rational Urysohn sphere have automatic homeomorphicity with respect to the class of closed monoids on $\bN$ with finitely many weak orbits.

The usual way for showing automatic homeomorphicity of a closed transformation monoid $M$ goes as follows:
\begin{enumerate}
	\item show automatic homeomorphicity of the group $G$ of invertible elements in $M$,
	\item show automatic homeomorphicity of the closure of $G$ in $M$,
	\item show automatic homeomorphicity of $M$.
\end{enumerate}
For the first step, there is a wealth of results in the literature. Usually one can observe that $G$ has the small index property. In the third step so-called gate-techniques are employed (cf. \cite{BodPinPon17}, \cite{PecPec15b}). So far the hardest step has always been to lift automatic homeomorphicity from $G$ to the closure of $G$ in $M$. To cite \cite{BodPinPon17}:
\begin{quote}
	Surprisingly, it turns out to be a non-trivial task to show for a given closed oligomorphic subgroup $\mathbf{G}$ of $\mathbf{S}$ with automatic continuity (the strongest form of reconstruction) that the closure $\overline{\mathbf{G}}$ of $\mathbf{G}$ in $\mathbf{O}^{(1)}$ has some form of reconstruction.
\end{quote}
A good part of \cite{BodPinPon17} and of \cite{BehTruVar17} is dedicated to this problem, in order to arrive at the above mentioned results.

The problem of lifting automatic homeomorphicity from a closed permutation group $G$ to its closure in the full transformation monoid can be reduced to a purely algebraic sufficient condition:
\begin{proposition}[{\cite[Lemma 12]{BodPinPon17}}]\label{bppcrit}
	Let $A$ be a countable set, and let $M\le T_A$ be a closed transformation monoid whose group $G$ of invertible elements lies dense in $M$ and has automatic homeomorphicity. If the only injective endomorphism of $M$ that fixes every element of $G$ is the identity on $M$, then $M$ has automatic homeomorphicity.	
\end{proposition}

\section{Main results}
Before being able to formulate the main results of this paper, we need to introduce a few more notions from model theory.

Let $\lambda$ be any cardinal number. Recall that a structure $\bA$ is called \emph{$\lambda$-saturated} if
\[ \forall \kappa<\lambda\,\forall\bar{a}\in A^\kappa\,\forall \bB \succeq \bA\,\forall b\in B\,\exists c\in A:(\bB,\bar{a}b)\equiv(\bA,\bar{a}c).\]
Moreover, $\bA$ is called \emph{saturated} if it is $|A|$-saturated. 

Sometimes the following characterization of saturated structures is more convenient:
\begin{lemma}[{\cite[Lemma 10.1.3]{Hod93}}]\label{elemweakhom}
	Let $\bA$ be an $L$-structure and let $\lambda$ be a cardinal. Then $\bA$ is $\lambda$-saturated if and only if for every 	$L$-structure $\bB$ holds:
	\begin{multline*}
		\forall\kappa<\lambda\,\forall \ta\in A^\kappa\,\forall\tb\in B^\kappa: \\(\bA,\ta)\equiv(\bB,\tb)\Rightarrow \forall d\in B\,\exists c\in A : (\bA,\ta c)\equiv(\bB,\tb d).
	\end{multline*}
\end{lemma}

Saturated structures are rich in symmetries: Every saturated structure $\bA$ is strongly elementarily $|A|$-homogeneous. Here we call a structure $\bM$ \emph{strongly elementarily $\lambda$-homogeneous} if for all $\kappa<\lambda$ and for all $\ta, \tb\in M^\kappa$ with  $(\bM,\ta)\equiv(\bM,\tb)$, there exists an automorphism of $\bM$ that maps $\ta$ to $\tb$.

Consequently, if $\bA$ is a countable saturated structure then its automorphism group lies dense in the monoid of elementary self-embeddings of $\bA$ (from here on, we will denote this monoid by $\EEmb(\bA)$, while the monoid of self-embeddings will be denoted by $\Emb(\bA)$).

The first result of this paper is:
\begin{theorem}\label{authomeom}
	Let $\bA$ be a countable saturated structure such that $\Aut(\bA)$ has automatic homeomorphicity, and such that $\Aut(\bA)$ has a trivial center. Then $\EEmb(\bA)$ has automatic homeomorphicity, too.
\end{theorem}

Examples for countable saturated structures whose automorphism group has a trivial center are countable saturated structures with no algebraicity (cf. \cite[7.1]{Hod93}). It is well-known that all countable $\aleph_0$-categorical structures are saturated. 

More concrete examples can be found in the world of countable homogeneous structures:  Recall that  a structure $\bU$ is called \emph{homogeneous} if every embedding of a finitely generated substructure of $\bU$ into $\bU$ extends to an automorphism of $\bU$ (the embeddings of finitely generated substructures of $\bU$ into $\bU$ are commonly called \emph{local isomorphisms} of $\bU$). The class of finitely generated structures embeddable into $\bU$ is called \emph{age} of $\bU$ and is denoted by $\Age(\bU)$. A classical theorem by \Fraisse{} states that each countable homogeneous structure is determined up to isomorphism by its age. Moreover, a class of finitely generated structures of the same type is the age of a countable homogeneous structure if and only if it splits into countably many isomorphism classes, it has the hereditary property (HP), the joint embedding property (JEP), and the amalgamation property (AP) (cf. \cite[7.1]{Hod93}).  

For countable  saturated homogeneous structures, we can read off the age of the structure whether it has algebraicity, or not. In particular, a countable saturated homogeneous structure has  no algebraicity if and only if its age has the \emph{strong amalgamation property} (SAP) (cf. \cite[Theorem 7.1.8]{Hod93}). 

To put the finger onto some concrete structures, the following is a list of countable homogeneous structures, for which Theorem~\ref{authomeom} delivers that the monoid of (elementary) self-embeddings has automatic homeomorphicity (for each new example we give  a reference, where the small index property of the respective automorphism group was proved):
\begin{itemize}
	\item all the examples from \cite{BodPinPon17} and \cite{BehTruVar17} (i.e., $(\bN,=)$, the Rado graph, the countable universal homogeneous digraph, $(\bQ,\le)$,\dots),
	\item the countable atomless Boolean algebra (\cite{Tru89}),
	\item the rational Urysohn space and the rational Urysohn sphere (\cite{Sol05},\cite{KecRos07}),
	\item  $\omega$-stable, $\omega$-categorical structures whose automorphism groups have a trivial center (\cite{HodHodLasShe93}).
	\item the Henson graphs and the Henson digraphs (\cite{Her98}),
	\item the edge-colored random graphs with countably many colors (\cite{DolMas12}).
\end{itemize}

Now that we talked about positive examples, let us have a look onto the side of negative examples. Evans and Hewitt in \cite{EvaHew90}, basing on a construction by Hrushovski,  gave an example of two countable $\aleph_0$-categorical structures whose automorphism groups are isomorphic as abstract groups but not as topological groups. Consequently, non of the two automorphism groups has automatic homeomorphicity.

Recently, Bodirsky, Evans, Kompatscher, and Pinsker in \cite{BodEvaKomPin15}, showed for  these structures, that their monoids of elementary self-embeddings are isomorphic as abstract monoids, but not as topological monoids. In other words, neither of the monoids has automatic homeomorphicity. In some sense, their example maximally violates the conditions of Theorem~\ref{authomeom} in that none of the automorphism groups has automatic homeomorphicity, nor do they have a trivial center. 

Coming back to the positive side, we would like to mention a classic result by Lascar (here stated in a slightly weaker form):
\begin{theorem}[{\cite[Proposition principale]{Las89}}]\label{lascar}
	Let $\bA$ be a countable $G$-finite, $\aleph_0$-categorical structure, and let $\bB$ be another countable structure. Let $\varphi$ be a monoid isomorphism from $\EEmb(\bA)$ to $\EEmb(\bB)$. Then $\varphi\restr_{\Aut(\bA)}$ is a topological isomorphism from $\Aut(\bA)$ to $\Aut(\bB)$.  
\end{theorem}
Here, a countable $\aleph_0$-categorical structure $\bA$ is called \emph{$G$-finite} if its automorphism group has a smallest closed subgroup of finite index. Note that the examples from \cite{BodEvaKomPin15} are not $G$-finite. 

The second  result of this paper is going to be a strengthening of Lascar's result (with a slight restriction of generality):
\begin{theorem}\label{genLascar}
	Let $\bA$ be a countable $G$-finite, $\aleph_0$-categorical structure whose automorphism group has a trivial center, and let $\bB$ be another countable structure. Let $\varphi$ be a monoid isomorphism from $\EEmb(\bA)$ to $\EEmb(\bB)$. Then $\varphi$ is a homeomorphism.
\end{theorem}

The proofs of  Theorem~\ref{authomeom} and Theorem~\ref{genLascar} will hinge on the following proposition:
\begin{proposition}\label{endstab}
	Let $\bA$ be a countable saturated structure such that $\Aut(\bA)$ has a trivial center. Then every endomorphism of $\EEmb(\bA)$ that fixes $\Aut(\bA)$ element-wise, is the identity on $\EEmb(\bA)$. 
\end{proposition}
Once this proposition will be proved, the proof of Theorem~\ref{authomeom} is immediate, using Proposition~\ref{bppcrit}.

The proof of Theorem~\ref{genLascar} goes roughly in the same way as the proof of Proposition~\ref{bppcrit}:
\begin{proof}[Proof of Theorem~\ref{genLascar}]
	Let $\psi:=\varphi\restr_{\Aut(\bA)}$. By Theorem~\ref{lascar} we have that  $\psi\colon\Aut(\bA)\to\Aut(\bB)$ is a homeomorphism. It is well-known (cf.~\cite[Proposition 11]{BodPinPon17}) that $\psi$ extends to a homeomorphism $\overline\psi\colon \EEmb(\bA)\to\overline{\Aut(\bB)}\le\EEmb(\bB)$.  
	 Let $h:=\varphi^{-1}\circ\overline\psi$. Then $h$ is an endomorphism of $\EEmb(\bA)$ that fixes $\Aut(\bA)$ pointwise. Hence, by Proposition~\ref{endstab}, $h$ is the identity endomorphism. It follows that $\varphi$ is equal to $\overline\psi$. In particular, $\varphi$ is a homeomorphism.
\end{proof}

Note that in contrast to Theorem ~\ref{authomeom}, Theorem~\ref{genLascar} seems not to  assume automatic homeomorphicity for the automorphism group of the structure $\bA$. However, this assumption might just be hidden. Lascar conjectured in \cite{Las91}, that the automorphism groups of all  $\aleph_0$-categorical $G$-finite structures have the small index property and hence automatic homeomorphicity. 

In order to prove Proposition~\ref{endstab}, we will use a strategy that was established in \cite{BodPinPon17}, and adapted in \cite{BehTruVar17}. It goes as follows: First a class of special elements of $\EEmb(\bA)$ is identified. Then it is shown that the special elements are fixed by every endomorphism of $\EEmb(\bA)$ that fixes the automorphisms of $\bA$. Finally it is shown that the image of the non-special elements is determined by the images of the special elements. In \cite{BodPinPon17}, the special elements were self-embeddings with rich and co-rich image. In \cite{BehTruVar17}, the special self-embeddings of $(\bQ,\le)$ are those self-embeddings with a maximally spread out image. In the next section we are going to introduce a new class of special morphisms: self-embeddings with superhomogeneous image. These are are generalizations of self-embeddings with maximally spread out image and are somewhat related to the conjugable endomorphisms introduced in \cite{Bar2015}.

The following easy observation about saturated structures will simplify many considerations in the sequel of this paper:
\begin{lemma}\label{smooth}
	Let $\bM$ be a countable saturated structure. Then there exists a  saturated structure $\widetilde\bM$ on the same set $M$ over a countable signature $L$, such that $\Aut(\bM)=\Aut(\widetilde\bM)$ and such that $\Th(\widetilde\bM)$ has quantifier-elimination. 
\end{lemma}
\begin{proof}
	It is easy to see that $\bM$ is definitionally equivalent to a structure $\bM'$ over a countable language (cf. \cite[Exercise 10.2.20]{Hod93}). Now we obtain $\widetilde\bM$ by expanding $\bM'$ with all definable relations (there are just countably many, as the signature of $\bM'$ is countable). Clearly, the theory of $\widetilde\bM$ has quantifier elimination. Moreover, the class of $\lambda$-saturated structures is stable under definitional expansions and reducts (cf. \cite[Exercises 10.2.11,10.2.12]{Hod93}). It follows that $\widetilde\bM$ is $\aleph_0$-saturated. 
\end{proof}
We are going to call a countable saturated structure \emph{smooth} if its signature is countable and if its theory admits quantifier-elimination. By Lemma~\ref{smooth}, if we can prove Proposition~\ref{endstab} for smooth countable saturated structures, then we have shown it also in the general case. Smooth countable saturated structures have also the pleasant property to be homogeneous,  and model complete, i.e., every embedding to an elementary equivalent structure is already an elementary embedding. This is going to make our life much easier.

\section{Superhomogeneous substructures}
\begin{definition}
	Let $\bU$ be a structure. A substructure $\bV\le\bU$ is called \emph{superhomogeneous in $\bU$} if 
	\begin{enumerate}
		\item every local isomorphism of $\bV$ extends to an automorphism of $\bU$ whose restriction to $V$ is an automorphism of $\bV$, and if
		\item for all $y\in U\setminus V$ there exists some $\alpha\in\Aut(\bU)$ that fixes $V$ pointwise such that $\alpha(y)\neq y$.
	\end{enumerate}
 \end{definition}
Note how the first part of the definition assures that the substructure has enough inner symmetries, while the second condition asserts that at the same time the pointwise stabilizer of $V$ in $\Aut(\bU)$ is  sufficiently rich.

Of course, superhomogeneous substructures are always homogeneous. In our applications of this concept also the enveloping structure will be homogeneous. In fact it will usually be isomorphic to the superhomogeneous substructure. To see that we are not talking about an empty concept, let us have a look onto some examples:
\begin{example}
	The clopen intervals in $(\bQ,\le)$ are superhomogeneous in $(\bQ,\le)$. Indeed, let $a,b\in(\bR\cup\{-\infty,+\infty\})\setminus\bQ$, and let $I=(a,b)\cap \bQ$. Then $(I,\le)\cong(\bQ,\le)$ and every automorphism of $(I,\le)$ extends to an automorphism of $(\bQ,\le)$ that leaves $I$ invariant. Moreover, the pointwise stabilizer of $I$ in $\Aut(\bQ)$ moves every point outside of $I$.
\end{example}

\begin{example}
	Another example are the maximally spread out substructures of $(\bQ,\le)$ introduced in \cite{BehTruVar17} for showing that $\Emb(\bQ,\le)$ has automatic homeomorphicity: Consider the $2$-colored rationals $(\bQ_2,\le)$. They are obtained from $(\bQ,\le)$ by assigning one of two colors (say, red and blue) to every rational in such a way that the red rationals and the blue rationals both form dense unbounded chains that lie dense in each other in the sense that between any two blue rationals there is a red one and between any two red rationals there is a blue one.
	Now replace every blue rational by a blue copy of the rationals. The resulting chain is again isomorphic to the rationals. It can be checked that the red rationals are superhomogeneous in this new version of the rationals. 
\end{example}

The following proposition will make clear, why we are interested into superhomogeneous substructures:
\begin{proposition}\label{superhomfix}
	Let $\bU$ be a countable homogeneous structure such that $\Aut(\bU)$  has a trivial center. Let $\iota$ be an  endomorphism of $\Emb(\bU)$ that fixes $\Aut(\bU)$ element-wise. Let $h\in\Emb(\bU)$ such that the image of $h$ is superhomogeneous in $\bU$. Then $\iota(h)=h$.  
\end{proposition}
Before coming to the proof, let us prepare the terrain with some auxiliary results:
\begin{definition}
	Let $G\le \Sym(A)$ be a permutation group and let $M\le T_A$ be the closure of $G$ in $T_A$. For every $f\in M$, $x\in A$ we define
	\begin{align*}
	f^*&:=\{(\alpha,\beta)\mid \alpha,\beta\in G,\alpha \circ f = f\circ\beta\},\\
	f^*(x)&:=\bigcap\{\fix(\alpha)\mid (\alpha,\beta)\in f^*, x\in\fix(\beta)\},\\
	I(f) &:= \bigcup_{x\in A} f^*(x),\\
	B(f) &:=\{\beta\in\Aut(\bA)\mid \exists\alpha\in\Aut(\bA):(\alpha,\beta)\in f^*\}.
	\end{align*}
	Here, as usual, $\fix(\pi)$ denotes the set of fixed points of a permutation $\pi$. 
\end{definition}

\begin{lemma}\label{xtwostar}
	With the notions from above, if $(\alpha,\beta)\in f^*$, then $\alpha$ preserves $\im(f)$. 
\end{lemma}
\begin{proof}
	Let $y\in \im(f)$---say, $y=f(x)$. Then 
	\[
	\alpha(y)=\alpha(f(x))=(\alpha\circ f)(x) = (f\circ\beta)(x)=f(\beta(x))\in\im(f).\qedhere
	\]
\end{proof}
\begin{lemma}\label{plus}
	With the notions from above, let $g:=f\circ\zeta$, for some $\zeta\in G$. Let $(\alpha,\beta)\in f^*$. Then $(\alpha,\beta)\in g^*$ if and only if $\beta\circ\zeta=\zeta\circ\beta$. 	
\end{lemma}
\begin{proof}
	\begin{align*}
		(\alpha,\beta)\in g^* &\iff \alpha\circ f\circ\zeta= f\circ\zeta\circ\beta \iff f\circ\beta\circ\zeta = f\circ\zeta\circ\beta\\
		&\iff \beta\circ\zeta=\zeta\circ\beta.
	\end{align*}
\end{proof}

\begin{lemma}\label{xstar}
	Let $\bA$ be a structure, $f\in\Emb(\bA)$. Let $\alpha\in\Aut(\bA)$, such that  the restriction of $\alpha$ to $\im(f)$ is an automorphism of $\langle \im(f)\rangle_\bA$. Then there exists some $\beta\in\Aut(\bA)$, such that $(\alpha,\beta)$ and $(\alpha^{-1},\beta^{-1})$ both belong to $f^*$. 
\end{lemma}
\begin{proof}
	Let us factorize $f$ according to
	\[
	\begin{tikzcd}
		\bA \rar{f_{\cong}}[swap]{\cong}& \langle\im(f)\rangle_\bA\rar[hook]{\iota}[swap]{=} & \bA.
	\end{tikzcd}
	\]
	Let $\alpha'$ be the restriction of $\alpha$ to $\im(f)$. Define $\beta:= f_{\cong}^{-1}\circ\alpha'\circ f_{\cong}$. Then $\beta\in\Aut(\bA)$. Now it remains to compute:
	\begin{align*}
		f\circ\beta &= \iota\circ f_{\cong}\circ f_{\cong}^{-1}\circ\alpha'\circ f_{\cong} = \iota\circ\alpha'\circ f_{\cong} = \alpha\circ f,\\
		f\circ\beta^{-1} &= \iota\circ f_{\cong}\circ f_{\cong}^{-1}\circ\alpha'^{-1}\circ f_{\cong} = \iota\circ\alpha'^{-1}\circ f_{\cong} = \alpha^{-1}\circ f.
	\end{align*}
\end{proof}

\begin{lemma}\label{Iimf}
	With the notions from above, $\im(f)\subseteq I(f)$ for all $f\in M$. 
\end{lemma}
\begin{proof}
	First of all note that $f^*$ is never empty, as it contains $(1_A,1_A)$. Let $x\in A$, $(\alpha,\beta)\in f^*$. Suppose $\beta(x)=x$. Then $\alpha(f(x))= (\alpha\circ f)(x)= (f\circ\beta)(x)= f(\beta(x))=f(x)$. Thus, $f(x)\in f^*(x)\subseteq I(f)$.
\end{proof}

\begin{lemma}\label{Bdense}
	Let $\bA$ be a countable homogeneous structure and let $f\in\Emb(\bA)$. If $\im(f)$ is superhomogeneous in $\bA$, then $B(f)$	lies dense in $\Aut(\bA)$. 
\end{lemma}
\begin{proof}
	Let $\lambda\colon\bC\injto\bA$ be a local isomorphism of $\bA$. Then $f\circ\lambda\circ f^{-1}\colon \langle f(C)\rangle_\bA\injto\langle \im(f)\rangle_\bA$ is a local isomorphism of $\langle \im(f)\rangle_\bA$. Since $\langle\im(f)\rangle_\bA$ is superhomogeneous in $\bA$, there exists $\alpha\in\Aut(\bA)$ that extends $f\circ\lambda\circ f^{-1}$, such that $\alpha\restr_{\im(f)}$ is an automorphism of $\langle\im(f)\rangle_\bA$. By Lemma~\ref{xstar}, there exists some $\beta\in\Aut(\bA)$, such that $(\alpha,\beta)\in f^*$. We claim that $\beta$ extends $\lambda$: Let $y\in f(C)$, say, $y=f(x)$. Then
	\[
	f(\beta(x)) = \alpha(f(x)) = \alpha(y) = f(\lambda(f^{-1}(y))) = f(\lambda(x)).
	\] 
	Since $f$ is injective, we conclude that $\lambda(x)=\beta(x)$. 
	
	Thus, every local isomorphism of $\bA$ extends to an element of $B(f)$. It follows that $B(f)$ is dense in $\Aut(\bA)$.
\end{proof}

\begin{lemma}\label{Iimfeq}
	Let $\bA$ be a countable homogeneous structure and let $f\in\Emb(\bA)$. If $\langle\im(f)\rangle_\bA$ is superhomogeneous in $\bA$, then $I(f)=\im(f)$.
\end{lemma}
\begin{proof}
Let $y\in I(f)$, say, $y\in f^*(x)$ for some $x\in A$. Suppose that $y\notin\im(f)$. Then, since $\langle\im(f)\rangle_\bA$ is superhomogeneous in $\bA$, there exists $\alpha\in\Aut(\bA)$  that fixes $\im(f)$ pointwise, but that moves $y$. Let $\beta:=1_A$. Then for every $z\in A$ we compute
\[
  (f\circ\beta)(z) = f(\beta(z))=f(z)=\alpha(f(z))= (\alpha\circ f)(z).
\]
In other words, $(\alpha,\beta)\in f^*$. But now, by construction,  $\beta(x)=x$ and $\alpha(y)\neq y$. Hence $y\notin f^*(x)$---a contradiction. Thus, $I(f)\subseteq \im(f)$. With Lemma~\ref{Iimf} the claim follows.
\end{proof}

\begin{corollary}\label{twoplus}
	Let $\bU$ be a countable homogeneous structure, and let $\iota$ be an endomorphism of $\Emb(\bU)$ that fixes $\Aut(\bU)$ pointwise. Let $h\in\Emb(\bU)$ such that $\langle\im(h)\rangle_\bU$ is superhomogeneous in $\bU$. Then $\im(h)=\im(\iota(h))$.  
\end{corollary}
\begin{proof}
	Clearly, $h^*=(\iota(h))^*$. Hence, using Lemma~\ref{Iimfeq}, we have $I(h)=I(\iota(h))=\im(h)$. By Lemma~\ref{Iimf} we have $\im(\iota(h))\subseteq\im(h)$. In particular, $\iota(h)$ preserves $\im(h)$. 
	
	Suppose that $\im(\iota(h))\subsetneq\im(h)$. Let $y\in\im(h)\setminus\im(\iota(h))$, say, $y=h(x)$. Then $y\neq \iota(h)(y)\in\im(\iota(h))$. As $\iota(h)$ preserves $\im(h)$, it follows that the mapping $y\mapsto\iota(h)(y)$ is a local isomorphism of $\langle\im(h)\rangle_\bU$. Thus, since $\langle\im(h)\rangle_\bU$ is superhomogeneous in $\bU$, there exists some $\alpha\in\Aut(\bU)$ that extends this to an automorphism of $\bU$ whose restriction to $\im(h)$ is an automorphism  of $\langle\im(h)\rangle_\bU$. By Lemma~\ref{xstar} there exists $\beta\in\Aut(\bU)$, such that $(\alpha^{-1},\beta^{-1})\in h^*=(\iota(h))^*$. Since $(\alpha^{-1},\beta^{-1})\in (\iota(h))^*$, it follows from Lemma~\ref{xtwostar} that $\alpha^{-1}$ preserves $\im(\iota(h))$. However, $\alpha^{-1}(\iota(h)(y)) = y\notin\im(\iota(h))$---a contradiction. Thus, $\im(h)=\im(\iota(h))$. 
\end{proof}

\begin{proposition}\label{threeplus}
	Let $\bU$ be a countable homogeneous structure and let $f,g$ be self-embeddings of $\bU$ with the same superhomogeneous image. Then
	\[
	   f^*=g^* \iff\exists\zeta\in Z(\Aut(\bU)): g=f\circ\zeta.
	\]
\end{proposition}
\begin{proof}
	``$\Rightarrow$'' Suppose $f^*=g^*$. Since $\im(f)=\im(g)$, there exists $\zeta\in\Aut(\bU)$, such that $g= f\circ\zeta$. 
	Since $f^*=g^*$, by Lemma~\ref{plus}, we have that $\beta\circ\zeta=\zeta\circ\beta$, for all $\beta\in B(f)$. Since $\langle\im(f)\rangle$ is superhomogeneous in $\bA$, it follows from Lemma~\ref{Bdense}, that $B(f)$ lies dense in $\Aut(\bA)$. Thus, $\zeta$ commutes with a dense subset of $\Aut(\bA)$. It follows, that $\zeta$ is in the center of $\Aut(\bA)$. 
	
	``$\Leftarrow$'' Suppose that there exists some $\zeta\in Z(\Aut(\bA))$, such that $g=f\circ\zeta$. Let $\alpha, \beta\in\Aut(\bA)$. Then 
	\begin{align*}
	(\alpha,\beta)\in f^* &\iff \alpha\circ f = f\circ\beta \iff \alpha\circ f\circ\zeta = f\circ\beta\circ\zeta\\
	&\iff\alpha\circ f\circ\zeta = f\circ\zeta\circ\beta \iff \alpha\circ g = g\circ\beta\\
	&\iff (\alpha,\beta)\in g^*.
	\end{align*}
\end{proof}

\begin{proof}[Proof of Proposition~\ref{superhomfix}]
	Since $\langle\im(h)\rangle_\bU$ is superhomogeneous in $\bU$, it follows from Corollary~\ref{twoplus}, that  $h$ and $\iota(h)$ have the same image. Moreover, from Proposition~\ref{threeplus} it follows that $\iota(h)=h\circ\zeta$, for some $\zeta\in Z(\Aut(\bU)$. Since the center of $\Aut(\bU)$ is trivial, it follows  that $h=\iota(h)$.
\end{proof}

\section{Existence of superhomogeneous substructures}
In this section we are going to show that smooth countable saturated structures are rich in selfembeddings with superhomogeneous image. 

The main result of this section is going to be:
\begin{proposition}\label{alphabeta}
	Let $\bA$ be a smooth countable saturated structure, and let $\bB$ be a substructure of $\bA$ isomorphic to $\bA$. Then there exists an extension $\widetilde\bA$ of $\bA$, such that $\widetilde\bA\cong\bA$ and such that both, $\bA$ and $\bB$, are superhomogeneous in $\widetilde\bA$.  
\end{proposition}
For proving this proposition we need to invoke somewhat heavier tools from model theory---highly saturated structures. Most of the auxiliary results needed for our proof can be found in the standard model-theoretic literature (see, e.g., Chapter 10 of \cite{Hod93}). For the convenience of the reader we collect the necessary results with exact references before going on to prove Proposition~\ref{alphabeta}.

\begin{definition}
	Let $L$ be a signature. An $L$-structure $\bM$ is called \emph{splendid} if for all extensions $L^+$ of $L$ by one relational symbol $\varrho$ and for every $L^+$-structure $\mathbf{N}$ whose $L$-reduct is elementarily equivalent with $\bM$ there exists an interpretation $\varrho^{\bM}$ of $\varrho$ in $\bM$ such that $(\bM, \varrho^{\bM})\equiv \mathbf{N}$. The structure $\bM$ is called \emph{$\lambda$-big} if $(\bM,\ta)$ is splendid for every tuple $\ta$ over $M$ of length $<\lambda$. 
\end{definition}

\begin{theorem}[{\cite[Theorem 10.2.1]{Hod93}}]\label{lambdabig}
	Let $L$ be a signature, $\bA$ be an $L$-structure, $\lambda$ be a regular cardinal such that $\lambda > |L|$. Then $\bA$ has a $\lambda$-big elementary extension $\widehat\bA$ of cardinality $\le |A|^{<\lambda}$.
\end{theorem}

\begin{corollary}\label{aleph1big}
	Let $L$ be a countable signature and let $\bA$ be a countable $L$-structure. Then $\bA$ has an $\aleph_1$-big elementary extension $\widehat\bA$ such that $|\widehat{A}|\le 2^{\aleph_0}$
\end{corollary}
\begin{proof}
	$\aleph_1$ is a successor cardinal, and thus it is regular (cf.~\cite[Corollary 5.3]{Jec03}). So by Theorem~\ref{lambdabig}, $\bA$ has an $\aleph_1$-big elementary extension $\widehat\bA$ with $|\widehat{A}|\le |A|^{<\aleph_1}\le \aleph_0^{<\aleph_1}=\sup\{\aleph_0^\kappa\mod \kappa \text{ cardinal}, \kappa<\aleph_1\}=\aleph_0^{\aleph_0}=2^{\aleph_0}$.
\end{proof}

\begin{proposition}[{\cite[Theorem 10.1.2, Exercise 10.1.4(a)]{Hod93}}]\label{elemhom}
	Every $\lambda$-big structure is $\lambda$-saturated and strongly elementarily $\lambda$-homoge\-neous. 
\end{proposition}

\begin{definition}
	Let $\lambda$ be a cardinal. A structure $\bM$ is called $\lambda$-homogeneous if for all $\kappa<\lambda$, and for all $\ta,\tb\in M^\kappa$ we have
	\[ (\bM,\ta)\equiv_0(\bM,\tb) \Rightarrow \exists\alpha\in\Aut(\bM): \alpha(\ta)=\tb.\]
\end{definition}
Note that a structure is $aleph_0$-homogeneous if and only if it is homogeneous.

\begin{lemma}\label{lambdahom}
	Let $\bA$ be a strongly elementarily $\lambda$-homogeneous structure whose theory has quantifier elimination. Then $\bA$ is $\lambda$-homogeneous.
\end{lemma}
\begin{proof}
	Let $\kappa<\lambda$, $\ta,\tb\in M^\kappa$, such that $(\bM,\ta)\equiv_0(\bM,\tb)$. Since $\Th(\bM)$ has quantifier-elimination, it follows that $(\bM,\ta)\equiv(\bM,\tb)$. Since $\bM$ is strongly elementarily $\lambda$-homogeneous, it follows that there exists $\alpha\in\Aut(\bM)$ that maps $\ta$ to $\tb$.
\end{proof}

Recall that an element $a$ of a structure $\bA$ is said to be \emph{definable} in $\bA$ over some set of parameters $X$ from A, if there exists a formula $\varphi$ with parameters from $X$ such that 
\[
	\bA\models \varphi(a)\land \exists_{=1}x\, \varphi(x).
\]
\begin{lemma}[{\cite[Exercise 10.1.5(b)]{Hod93}}]\label{twostar}
	Let $\lambda$ be an infinite cardinal, let $\bA$ be a $\lambda$-saturated structure, and let $X\subseteq A$ with $|X|<\lambda$. If an element $a\in A$ is not definable in $\bA$ over $X$, then at least two elements of $A$ realize $\tp_\bA(a/X)$.
\end{lemma}

\begin{lemma}\label{star}
	Let $\bU$ be a countable saturated structure and for some $\lambda>\aleph_0$ let  $\widehat\bU$ be a $\lambda$-saturated elementary extension of $\bU$. Then no element of $\widehat{U}\setminus U$ is definable over $U$.
\end{lemma}
\begin{proof}
	Suppose, on the contrary, that $\widehat{U}\setminus U$ contains an element $u$ that is definable over $U$. In other words, there exists a formula $\varphi$ with parameters in $U$ such that 
	\[
	\widehat{\bU}\models\varphi(u)\land\exists_{=1}x\,\varphi(x).
	\]
	In $\varphi$ only finitely many elements from $U$ are used, say, $a_1,\dots,a_n$. Let $\ta:=(a_1,\dots,a_n)$. Then we can write $\varphi(x)\equiv \psi(\ta,x)$ for some formula $\psi\equiv\psi(\ty,x)$ that does not contain any parameter from $U$. Since $\bU$ is $\aleph_0$-saturated and since $\widehat\bU$ is an elementary extension of $\bU$,  there exists $v\in U$ such that $(\widehat\bU,\ta u) \equiv (\bU,\ta v)$. In particular, $\bU\models\psi(\ta,v)$. Using again that $\widehat\bU\succeq\bU$, we conclude that we have $\widehat\bU\models\psi(\ta,v)$. But this is a contradiction with $\widehat\bU\models\exists_{=1}x \,\psi(\ta,x)$.  
\end{proof}

\begin{corollary}\label{stabauto}
	Let $\bU$ be a countable saturated structure and for some sufficiently large uncountable cardinal $\lambda$ let $\widehat\bU$ be a $\lambda$-big elementary extension of $\bU$. Then for every $u\in\widehat{U}\setminus U$ there exists an automorphism $\alpha$ of $\widehat\bU$ that fixes $U$ pointwise, but that does not fix $u$.
\end{corollary}
\begin{proof}
	By Proposition~\ref{elemhom}, $\widehat\bU$ is    $\lambda$-saturated. Let $u\in \widehat{U}\setminus U$. Then, by Lemma~\ref{star}, $u$ is not definable over $U$. Hence, by Lemma~\ref{twostar}, there exists $u'\in\widehat{U}$ distinct from $u$, such that $\tp_{\widehat\bU}(u/U)=\tp_{\widehat\bU}(u'/U)$.
	If $\ta$ is an enumeration of the elements of $U$, then this can be rewritten as $(\widehat\bU,\ta u)\equiv(\widehat\bU,\ta u')$. Since, by Proposition~\ref{elemhom}, $\widehat\bU$ is strongly elementarily $\lambda$-homogeneous, it follows that there exists an automorphism $\alpha$ of $\widehat\bU$ that maps $\ta u$ to $\ta u'$. In other words, $\alpha$ fixes $U$ pointwise and moves $u$ to $u'$.
\end{proof}

The following lemma is the last technical tool we need to prepare before we can finally prove Proposition~\ref{alphabeta}:
\begin{lemma}\label{sameage}
	Let $\bM$ be an $\aleph_0$-saturated structure. Let $\mathbf{N}$ be any other model of $\Th(\bM)$.  Then $\Age(\mathbf{N})\subseteq\Age(\bM)$.
\end{lemma}
\begin{proof}
	Let $\bA\in\Age(\mathbf{N})$. Without loss of generality, $\bA\le\mathbf{N}$, $\bA=\langle a_1,\dots,a_n\rangle_{\mathbf{N}}$. Using $n$ times that $\bM$ is $\aleph_0$-saturated (in particular, using $n$ times Lemma~\ref{elemweakhom}), we conclude that there exist $b_1,\dots,b_n\in M$ such that $(\mathbf{N},(a_1,\dots,a_n))\equiv(\bM,(b_1,\dots,b_n))$. In particular we have $\langle b_1,\dots,b_n\rangle_\bM\cong\bA$. Hence $\bA\in\Age(\bM)$.
\end{proof}

\begin{proof}[Proof of Proposition~\ref{alphabeta}.]
	Since $\bA$ is smooth, it has a countable signature. Thus, by Corollary~\ref{aleph1big}, it has an $\aleph_1$-big elementary extension of cardinality $\le 2^{\aleph_0}$. Let $\widehat\bA$ be such an extension. Also, from smoothness, using Proposition~\ref{elemhom} together with Lemma~\ref{lambdahom}, it follows that $\bA$ and $\bB$ are homogeneous and that $\widehat\bA$ is $\aleph_1$-homogeneous.
	
	The structure $\widetilde\bA$ is going to be constructed as the union of a tower of countable superstructures of $\bA$ in $\widehat\bA$:
	\[
	\widetilde\bA:=\bigcup_{i<\omega}\bA_i\qquad(\bA_0\le\bA_1\le\dots)
	\] 
	As the construction is very long and the work is very monotonous, we are going to program five robots for the job:
	\begin{description}
		\item[Robot 1] has to make sure that every local isomorphism of $\bB$ is going to be extendable to an automorphism of $\widetilde\bA$ whose restriction to $B$ is an automorphism of $\bB$.
		\item[Robot 2] has to make sure that every local isomorphism of $\bA$ is going to be extendable to an automorphism of $\widetilde\bA$ whose restriction to $A$ is an automorphism of $\bA$.
		\item[Robot 3] has to make sure that for every element $b\in\widetilde{A}\setminus B$ there is going to exist an automorphism of $\widetilde\bA$ that fixes $B$ pointwise but that does not fix $b$.
		\item[Robot 4] has to make sure that for every element $a\in\widetilde{A}\setminus A$ there is going to exist an automorphism of $\widetilde\bA$ that fixes $A$ pointwise but that does not fix $a$.
		\item[Robot 5] has to make sure that $\widetilde\bA$ is going to be homogeneous.
	\end{description}
	As was said before, the robots build up the tower $(\bA_i)_{i<\omega}$. In parallel they construct a tower $(G_i)_{i<\omega}$ of countable subgroups of $\Aut(\widehat\bA)$ such that $\bA_i$ is invariant under $G_i$, for all $i<\omega$. 
	
	$\bA_i$ and $G_i$ are going to be constructed in the $i$-th step of the construction. If a robot gets active in the $(i+1)$-st step, it knows everything that has been constructed before, and constructs $\bA_{i+1}$ and $G_{i+1}$. 
	
	Robots $1$ and $2$ will only work in the first and the second step, respectively. After that robots $3$, $4$, and $5$ will take turns. In the very beginning we define $\bA_0:=\bA$ and $G_0:=\{1_{\widehat\bA}\}$

	\begin{description}
		\item[Work of robot 1] Robot 1 takes a countable dense subgroup of $\Aut(\bB)$ and extends every element to an automorphism of $\widehat\bA$ (this is possible, since $\widehat\bA$ is $\aleph_1$-homogeneous). Then it adjoins all obtained automorphisms to $G_0$ to obtain $G_1$. Finally, it defines
		\[ 
		A_1':=\bigcup_{g\in G_1} g(A_0)\qquad\text{and}\qquad \bA_1:=\langle A_1'\rangle_{\widehat\bA}.
		\]
		As $G_1$ and $A_0$ are countable, so is $A_1'$. Finally, $\bA_1$ is countable, since the signature of $\widehat\bA$ is countable. Clearly, $A_1$ is invariant under $G_1$.
		\item[Work of robot 2] Robot 2 takes a countable dense subgroup of $\Aut(\bA)$ and extends every element to an automorphism of $\widehat\bA$. All the extensions are then adjoined to $G_1$ to obtain $G_2$. Finally,
		\[
			A_2':=\bigcup_{g\in G_2} g(A_1)\qquad\text{and}\qquad \bA_2:=\langle A_2'\rangle_{\widehat\bA}.
		\]
		\item[Work of robot 3] Suppose, robot 3 is active in step $(i+1)$. In particular, $\bA_i$ and $G_i$ are already constructed. Now, robot 3 collects all such points $x\in A_i\setminus B$ for which the pointwise stabilizer of $B$ in $G_i$ fixes $x$. Since $A_i$ is countable, only countably many elements are collected. For every such $x$, robot 3 chooses an automorphism of $\widehat\bA$ that fixes $B$ pointwise but that does not fix $x$. This is possible, since $\bB$ is countable saturated and smooth. Indeed, because $\bB\equiv\widehat\bA$, $\bB\le\widehat\bA$, and  because $\Th(\bB)$ has quantifier elimination, we have $\bB\preceq\widehat\bA$. Since $\widehat\bA$ is $\aleph_1$-big and since $\bB$ is countable saturated, it follows from Corollary~\ref{stabauto}, that such an automorphism exists. All obtained automorphisms of $\widehat\bA$ are adjoined to $G_i$ to obtain $G_{i+1}$. Finally,
		\[
		A_{i+1}':=\bigcup_{g\in G_{i+1}} g(A_i)\qquad\text{and}\qquad \bA_{i+1}:=\langle A_{i+1}'\rangle_{\widehat\bA}.
		\]
		\item[Work of robot 4] Robot 4 works analogously to robot 3. Instead of points from $A_i\setminus B$ it considers points from $A_i\setminus A$. 
		\item[Work of robot 5] Robot 5 adjoins for every local isomorphism of $\bA_i$ an extension to an automorphism of $\widehat\bA$ to $G_i$. Thus, it obtains $G_{i+1}$. Finally, as usually,
		\[
		A_{i+1}':=\bigcup_{g\in G_{i+1}} g(A_i)\qquad\text{and}\qquad \bA_{i+1}:=\langle A_{i+1}'\rangle_{\widehat\bA}.
		\]
	\end{description}
	Thus, the construction of the two towers is completed. Now we define
	\[
	\widetilde\bA:=\bigcup_{i<\omega} \bA_i, \qquad G:=\bigcup_{i<\omega} G_i.
	\]
	Our first observation is that $\widetilde{A}$ is invariant under $G$. Indeed, for all $x\in\widetilde{A}$, and $g\in G$ there is an $i<\omega$, such that $x\in A_i$, and $g\in G_i$. Since $A_i$ is invariant under $G_i$, the claim follows.
	
	Now we claim that $\widetilde\bA$ has the desired properties. Let $\alpha$ be a local isomorphism of $\bB$. Then $G_1$ contains an automorphism $\hat\alpha$ of $\widehat\bA$ that extends $\alpha$, and such that $\hat\alpha\restr_B$ is an automorphism of $\bB$. Since $\widetilde{A}$ is invariant under $G$, we obtain that $\tilde\alpha:=\hat\alpha\restr_{\widetilde{A}}$ is an automorphism of $\widetilde\bA$ that extends $\alpha$ such that $\tilde\alpha\restr_B$ is an automorphism of $\bB$.
	
	Let now  $y\in\widetilde{A}\setminus B$. Then $y\in A_i$, for some $i<\omega$. Without loss of generality, robot 3 is active in step $i$. But then, after step $i$, $G_{i+1}$ will contain an automorphism $\hat\alpha$ of $\widehat\bA$ that fixes $B$ pointwise, but does not fix $y$. We  may define $\tilde\alpha:=\hat\alpha\restr_{\widetilde{A}}$. Since $\widetilde{A}$ is invariant under $G$, we have that $\tilde\alpha$ is an automorphism of $\widetilde\bA$ that fixes $B$ pointwise and that does not fix $y$. Thus, $\bB$ is superhomogeneous in $\widetilde\bA$.
	
	Analogously, we argue that $\bA$ is superhomogeneous in $\widetilde\bA$.  
	
	It remains to show that $\widetilde\bA$ is isomorphic to $\bA$. First, since $\bA\le\widetilde\bA\le\widehat\bA$, we have $\Age(\bA)\subseteq\Age(\widetilde\bA)\subseteq\Age(\widehat\bA)$. Moreover, since $\bA$ is $\aleph_0$-saturated and elementarily equivalent with $\widehat\bA$, from Lemma~\ref{sameage} we conclude that $\Age(\widehat\bA)\subseteq\Age(\bA)$. Together this gives $\Age(\bA)=\Age(\widetilde\bA)$. As a smooth countable saturated structure $\bA$ is homogeneous. On the other hand, robot 5 assured that $\widetilde\bA$ is homogeneous, too. It follows from \Fraisse's theorem that $\bA$ and $\widetilde\bA$ are isomorphic.
\end{proof}

\section{Proofs of the main results}

\begin{proposition}\label{alphabetaex}
	Let $\bA$ be a smooth countable saturated structure. Then for every self-embedding $\iota$ of $\bA$ there exist self-embeddings $\alpha$ and $\beta$ of $\bA$, such that $\alpha=\beta\circ\iota$ and such that both, $\alpha$ and $\beta$,  have images superhomogeneous in $\bA$. 
\end{proposition}
\begin{proof}
	Let $\bB$ be the image of $\iota$. By Proposition~\ref{alphabeta}, $\bA$ has an extension $\widetilde\bA$, such that $\widetilde\bA\cong\bA$, and such that $\bA$ and $\bB$ are superhomogeneous in $\widetilde\bA$. Let $\lambda\colon\bA\injto\widetilde\bA$ be the identical embedding, and let $\kappa\colon\widetilde\bA\to\bA$ be any isomorphism:
	\[
	\begin{tikzcd}
		\bA \rar[swap]{\cong}\arrow[bend left,hook]{rr}{\iota} & \bB\rar[hook,swap]{=} & \bA\rar[hook]{\lambda}[swap]{=} & \widetilde\bA\rar{\kappa}[swap]{\cong} & \bA.
	\end{tikzcd}
	\]
	We define $\alpha:=\kappa\circ\lambda\circ\iota$, and $\beta:=\kappa\circ\lambda$. Then $\alpha=\beta\circ\iota$. Observe that $\im(\alpha)=\kappa(B)$. Since $\bB$ is superhomogeneous in $\widetilde\bA$ and since $\kappa$ is an isomorphism, we have also that the image of $\alpha$ induces a superhomogeneous substructure of $\bA$.  Observe further that $\im(\beta)=\kappa(A)$. By the same reasoning as above the image of $\beta$ induces a superhomogeneous substructure of $\bA$.
\end{proof}

\begin{proof}[Proof of Proposition~\ref{endstab}]
	By Lemma~\ref{smooth} we may assume without loss of generality that $\bA$ is smooth. In particular, $\bA$ is homogeneous and every self-embedding is elementary. Let $\iota$ be any endomorphism of $\Emb(\bA)$ that fixes $\Aut(\bA)$ pointwise. By Proposition~\ref{superhomfix}, $\iota$ fixes all self-embeddings of $\bA$ that have a superhomogeneous image. Let $h\in\Emb(\bA)$. By Proposition~\ref{alphabetaex}, there exist $\alpha,\beta\in\Emb(\bA)$ with superhomogeneous images in $\bA$, such that $\alpha=\beta\circ h$. But then we may compute:
	\[
	\beta\circ h = \alpha = \iota(\alpha)= \iota(\beta\circ h)=\iota(\beta)\circ \iota(h) = \beta\circ \iota(h).
	\]
	Since $\beta$ is injective, it follows that $h=\iota(h)$. In other words, $\iota$ is the identity on $\Emb(\bA)$. 
\end{proof}

%

\begin{thebibliography}{10}

\bibitem{Bar2015}
R.~Barham.
\newblock Automatic homeomorphicity of locally moving clones.
\newblock {\em ArXiv e-prints}, Dec. 2015.
\newblock URL: \url{http://arxiv.org/abs/1512.00251}, \href
  {http://arxiv.org/abs/1512.00251} {\path{arXiv:1512.00251}}.

\bibitem{BehTruVar17}
M.~Behrisch, J.~K. Truss, and E.~Vargas-Garc{\'\i}a.
\newblock Reconstructing the topology on monoids and polymorphism clones of the
  rationals.
\newblock {\em Studia Logica}, 105(1):65--91, 2017.
\newblock URL: \url{http://dx.doi.org/10.1007/s11225-016-9682-z}, \href
  {http://dx.doi.org/10.1007/s11225-016-9682-z}
  {\path{doi:10.1007/s11225-016-9682-z}}.

\bibitem{BodEvaKomPin15}
M.~Bodirsky, D.~Evans, M.~Kompatscher, and M.~Pinsker.
\newblock A counterexample to the reconstruction of $\omega$-categorical
  structures from their endomorphism monoids.
\newblock {\em ArXiv e-prints}, Oct. 2015.
\newblock URL: \url{http://arxiv.org/abs/1510.00356}, \href
  {http://arxiv.org/abs/1510.00356} {\path{arXiv:1510.00356}}.

\bibitem{BodPin15}
M.~Bodirsky and M.~Pinsker.
\newblock Topological {B}irkhoff.
\newblock {\em Trans. Amer. Math. Soc.}, 367(4):2527--2549, 2015.
\newblock \href {http://dx.doi.org/10.1090/S0002-9947-2014-05975-8}
  {\path{doi:10.1090/S0002-9947-2014-05975-8}}.

\bibitem{BodPinPon17}
M.~Bodirsky, M.~Pinsker, and A.~Pongr{\'a}cz.
\newblock Reconstructing the topology of clones.
\newblock {\em Trans. Amer. Math. Soc.}, 369(5):3707--3740, 2017.
\newblock URL: \url{http://dx.doi.org/10.1090/tran/6937}, \href
  {http://dx.doi.org/10.1090/tran/6937} {\path{doi:10.1090/tran/6937}}.

\bibitem{DixNeuTho86}
J.~D. Dixon, P.~M. Neumann, and S.~Thomas.
\newblock Subgroups of small index in infinite symmetric groups.
\newblock {\em Bull. London Math. Soc.}, 18(6):580--586, 1986.
\newblock \href {http://dx.doi.org/10.1112/blms/18.6.580}
  {\path{doi:10.1112/blms/18.6.580}}.

\bibitem{DolMas12}
I.~Dolinka and D.~Ma{\v s}ulovi{\'c}.
\newblock Properties of the automorphism group and a probabilistic construction
  of a class of countable labeled structures.
\newblock {\em Journal of Combinatorial Theory, Series A}, 119(5):1014 -- 1030,
  2012.
\newblock URL:
  \url{http://www.sciencedirect.com/science/article/pii/S0097316512000179},
  \href {http://dx.doi.org/10.1016/j.jcta.2012.01.008}
  {\path{doi:10.1016/j.jcta.2012.01.008}}.

\bibitem{EvaHew90}
D.~M. Evans and P.~R. Hewitt.
\newblock Counterexamples to a conjecture on relative categoricity.
\newblock {\em Ann. Pure Appl. Logic}, 46(2):201--209, 1990.
\newblock URL: \url{http://dx.doi.org/10.1016/0168-0072(90)90034-Y}, \href
  {http://dx.doi.org/10.1016/0168-0072(90)90034-Y}
  {\path{doi:10.1016/0168-0072(90)90034-Y}}.

\bibitem{Her98}
B.~Herwig.
\newblock Extending partial isomorphisms for the small index property of many
  {$\omega$}-categorical structures.
\newblock {\em Israel J. Math.}, 107:93--123, 1998.
\newblock \href {http://dx.doi.org/10.1007/BF02764005}
  {\path{doi:10.1007/BF02764005}}.

\bibitem{Hod93}
W.~Hodges.
\newblock {\em Model theory}, volume~42 of {\em Encyclopedia of Mathematics and
  its Applications}.
\newblock Cambridge University Press, Cambridge, 1993.
\newblock URL: \url{http://dx.doi.org/10.1017/CBO9780511551574}, \href
  {http://dx.doi.org/10.1017/CBO9780511551574}
  {\path{doi:10.1017/CBO9780511551574}}.

\bibitem{HodHodLasShe93}
W.~Hodges, I.~Hodkinson, D.~Lascar, and S.~Shelah.
\newblock The small index property for $\omega$-stable $\omega$-categorical
  structures and for the random graph.
\newblock {\em J. Lond. Math. Soc., II. Ser.}, 48(2):204--218, 1993.
\newblock \href {http://dx.doi.org/10.1112/jlms/s2-48.2.204}
  {\path{doi:10.1112/jlms/s2-48.2.204}}.

\bibitem{Jec03}
T.~Jech.
\newblock {\em Set theory. The third millennium edition, revised and expanded}.
\newblock Monographs in Mathematics. Springer 769 p., 2003.
\newblock \href {http://dx.doi.org/10.1007/3-540-44761-X}
  {\path{doi:10.1007/3-540-44761-X}}.

\bibitem{KecRos07}
A.~S. Kechris and C.~Rosendal.
\newblock Turbulence, amalgamation, and generic automorphisms of homogeneous
  structures.
\newblock {\em Proceedings of the London Mathematical Society}, 94(2):302--350,
  2007.
\newblock \href {http://dx.doi.org/10.1112/plms/pdl007}
  {\path{doi:10.1112/plms/pdl007}}.

\bibitem{Las89}
D.~Lascar.
\newblock Le demi-groupe des endomorphismes d'une structure
  $\aleph_0$-cat{\'e}gorique.
\newblock In M.~Giraudet, editor, {\em Actes de la Journ{\'e}e Alg{\`e}bre
  Ordonn{\'e}e (Le Mans, 1987)}, pages 33--43, 1989.

\bibitem{Las91}
D.~Lascar.
\newblock {Autour de la propri{\'e}t{\'e} du petit indice. (On the small index
  property).}
\newblock {\em Proc. Lond. Math. Soc., III. Ser.}, 62(1):25--53, 1991.
\newblock \href {http://dx.doi.org/10.1112/plms/s3-62.1.25}
  {\path{doi:10.1112/plms/s3-62.1.25}}.

\bibitem{Mac11}
D.~Macpherson.
\newblock A survey of homogeneous structures.
\newblock {\em Discrete Math.}, 311(15):1599--1634, 2011.
\newblock \href {http://dx.doi.org/10.1016/j.disc.2011.01.024}
  {\path{doi:10.1016/j.disc.2011.01.024}}.

\bibitem{PecPec15b}
C.~Pech and M.~Pech.
\newblock On automatic homeomorphicity for transformation monoids.
\newblock {\em Monatshefte f{\"u}r Mathematik}, 2015.
\newblock \href {http://dx.doi.org/10.1007/s00605-015-0767-y}
  {\path{doi:10.1007/s00605-015-0767-y}}.

\bibitem{Ros09b}
C.~Rosendal.
\newblock Automatic continuity of group homomorphisms.
\newblock {\em Bull. Symbolic Logic}, 15(2):184--214, 2009.
\newblock URL: \url{http://dx.doi.org/10.2178/bsl/1243948486}, \href
  {http://dx.doi.org/10.2178/bsl/1243948486}
  {\path{doi:10.2178/bsl/1243948486}}.

\bibitem{Sol05}
S.~Solecki.
\newblock Extending partial isometries.
\newblock {\em Israel J. Math.}, 150:315--331, 2005.
\newblock \href {http://dx.doi.org/10.1007/BF02762385}
  {\path{doi:10.1007/BF02762385}}.

\bibitem{Tru89}
J.~K. Truss.
\newblock Infinite permutation groups. {II}. {S}ubgroups of small index.
\newblock {\em J. Algebra}, 120(2):494--515, 1989.
\newblock \href {http://dx.doi.org/10.1016/0021-8693(89)90212-3}
  {\path{doi:10.1016/0021-8693(89)90212-3}}.

\bibitem{TruVar16}
J.~K. Truss and E.~Vargas-Garc{\'\i}a.
\newblock Reconstructing the topology on monoids and polymorphism clones of
  reducts of the rationals.
\newblock 06 2016.
\newblock URL: \url{https://arxiv.org/abs/1606.09531}, \href
  {http://arxiv.org/abs/1606.09531} {\path{arXiv:1606.09531}}.

\end{thebibliography}

\end{document}